\documentclass[a4paper,11pt]{amsart}
\usepackage{amssymb,amscd,amsxtra,xypic}
\usepackage[all]{xy}
\usepackage{array}

\usepackage[pdftex]{hyperref}
\hypersetup{%
bookmarksnumbered=true,%
colorlinks=true,%
setpagesize=false,%
pdftitle={},%
pdfauthor={Takuzo Okada}}

\theoremstyle{plain}
\newtheorem{Thm}{Theorem}[section]
\newtheorem{Lem}[Thm]{Lemma}

\newtheorem{Prop}[Thm]{Proposition}
\newtheorem{Conj}[Thm]{Conjecture}

\theoremstyle{definition}
\newtheorem{Def}[Thm]{Definition}
\newtheorem{Def-Lem}[Thm]{Definition-Lemma}

\newtheorem{Rem}[Thm]{Remark}

\newtheorem*{Ack}{Acknowledgments}
\newtheorem{Example}[Thm]{Example}

\newtheorem{Question}[Thm]{Question}

\makeatletter

\@addtoreset{equation}{section}
\makeatletter

\newcommand{\Sing}{\operatorname{Sing}}
\newcommand{\Spec}{\operatorname{Spec}}

\newcommand{\Cl}{\operatorname{Cl}}

\newcommand{\Pic}{\operatorname{Pic}}
\newcommand{\Int}{\operatorname{Int}}
\newcommand{\bNE}{\operatorname{\overline{NE}}}
\newcommand{\Mov}{\operatorname{Mov}}
\newcommand{\Bs}{\operatorname{Bs}}

\newcommand{\mult}{\operatorname{mult}}

\newcommand{\Pli}{\operatorname{Pli}}

\newcommand{\SL}{\operatorname{SL}}

\newcommand{\nef}{\mathrm{nef}}

\newcommand{\Cox}{\mathrm{Cox}}
\newcommand{\Nef}{\operatorname{Nef}}
\newcommand{\wtr}{\operatorname{wr}}

\newcommand{\mbA}{\mathbb{A}}

\newcommand{\mbC}{\mathbb{C}}

\newcommand{\mbP}{\mathbb{P}}
\newcommand{\mbQ}{\mathbb{Q}}
\newcommand{\mbR}{\mathbb{R}}

\newcommand{\mbZ}{\mathbb{Z}}

\newcommand{\mcF}{\mathcal{F}}

\newcommand{\mcH}{\mathcal{H}}
\newcommand{\mcI}{\mathcal{I}}

\newcommand{\mcO}{\mathcal{O}}

\newcommand{\msp}{\mathsf{p}}

\newcommand{\ratmap}{\dashrightarrow}

\makeatletter
\def\imod#1{\allowbreak\mkern10mu({\operator@font mod}\,\,#1)}
\makeatother

\title[Birational rigidity of del Pezzo fibrations]{On birational rigidity of singular del Pezzo fibrations of degree 1}
\author[Takuzo Okada]{Takuzo Okada}
\address{Department of Mathematics, Faculty of Science and Engineering\endgraf
Saga University, Saga 840-8502 Japan}
\email{okada@cc.saga-u.ac.jp}
\date{}

\begin{document}

\begin{abstract}
We give a sufficient condition for birational superrigidity of del Pezzo fibrations of degree $1$ with only $\frac{1}{2} (1,1,1)$ singular points, generalizing the so called $K^2$-condition.
As an application, we also prove that a del Pezzo fibrations of degree $1$ with only $\frac{1}{2} (1,1,1)$ singular points embedded in a toric $\mbP (1,1,2,3)$-bundle over $\mbP^1$ is birationally superrigid if and only if it satisfies the $K$-condition.
\end{abstract}

\maketitle


\section{Introduction} \label{sec:intro}

Birational rigidity of nonsingular del Pezzo fibrations of low degree is deeply studied by Pukhlikov and Grinenko.
A del Pezzo fibration $X \to \mbP^1$ over $\mbP^1$ is said to satisfy the $K^2$-{\it condition} (resp.\ $K$-{\it condition}) if the $1$-cycle $(-K_X)^2$ is not contained in the interior of the cone $\bNE (X)$ of effective curves on $X$ (resp.\ $-K_X$ is not in the interior of the movable cone $\Mov (X)$).
Note that the $K^2$-condition implies the $K$-condition.
Pukhlikov \cite{Puk} proved that a nonsingular del Pezzo fibration $X/\mbP^1$ of degree $1$, $2$ and $3$ (in the last case, a generality condition is required) satisfying the $K^2$-condition is birationally rigid.
Later on, Grinenko proved, in a series of papers \cite{Gri00,Gri03,Gri06}, that a nonsingular del Pezzo fibration of degree $1$ and $2$ (in the latter case, a generality condition is required) is birationally rigid (over the base) if and only if $X/\mbP^1$ satisfies the $K$-condition.
Based on these results, Grinenko proposed the following.

\begin{Conj}[Grinenko] \label{conj}
A del Pezzo fibration of degree $1$ (with at most terminal singularities) is birationally rigid over the base if and only if it satisfies the $K$-condition.
\end{Conj}

The ``only if" part is proved in \cite{Gri01}.
We remark that the same conjecture for del Pezzo fibrations of degree $2$ does not hold: Ahmadinezhed \cite{Ahm2} gave an example of a birationally non-rigid singular (Gorenstein) del Pezzo fibration of degree $2$ satisfying the $K$-condition.

In the context of Minimal Model Program, it is natural and important to study singular del Pezzo fibrations.
Recently there are some progress for singular del Pezzo fibrations of degree $2$:
Krylov \cite{Krylov} and Ahmadinezhad--Krylov \cite{AK} proved that a del Pezzo fibration of degree $2$ with only singular points of type $\frac{1}{2} (1,1,1)$ satisfying the $K^2$-condition is birationally rigid under some additional assumptions.

One of the main aims of this paper is to give a sufficient condition for birational (super)rigidity of singular del Pezzo fibrations of degree $1$.

\begin{Def}
Let $X/\mbP^1$ be a del Pezzo fibration of degree $1$.
We denote by $F$ the fiber class of the fibration $X \to \mbP^1$.
We define
\[
\nef (X/\mbP^1) := \inf \{\, r \mid \text{$-K_X + r F$ is nef}\,\},
\]
and call it the {\it nef threshold} of $X/\mbP^1$.
For a number $\delta \in \mbR$, we say that $X/\mbP^1$ satisfies the $K^3_{\delta}$-{\it condition} if the inequality
\[
(-K_X)^3 + \nef (X/\mbP^1) \le \delta
\]
is satisfied.
\end{Def}

\begin{Rem} \label{rem:Kcondsequiv}
We see that the $K^2$-condition is equivalent to the $K^3_{0}$-condition.
Indeed, since the nef cone of $X$ is spanned by $F$ and $-K_X + \nef (X/\mbP^1) F$, and $F \cdot (-K_X)^2 = 1 > 0$, the $1$-cycle $(-K_X)^2$ is not in the interior of $\bNE (X)$, i.e., the $K^2$-condition is satisfied, if and only if the inequality
\[
(-K_X + \nef (X/\mbP^1)F) \cdot (-K_X)^2 \le 0
\]
holds, which is nothing but the $K^3_0$-condition.
Moreover it is obvious that the $K^3_{\delta}$-condition implies the $K^3_{\delta'}$-condition for $\delta \le \delta'$.
Thus, for any $\delta \ge 0$, we have the implications
\[
\text{$K^2$-condition} \ \Longrightarrow \ \text{$K$-condition \& $K^3_{\delta}$-condition} \ \Longrightarrow \text{$K$-condition}.
\]
Moreover it is worth mentioning that Iskovskikh already noticed the potential importance of the invariant $(-K_X)^3 + \nef (X/\mbP^1)$ and it was conjectured in \cite{Iskovskikh} that a nonsingular del Pezzo fibrations of degree $1$ is birationally rigid if the inequality $(-K_X)^3 + \nef (X/\mbP^1) \le 1$ is satisfied.
\end{Rem}

We state the main theorem, where we refer readers to Definition \ref{def:birsuperrig} for the definition of birational superrigidity.

\begin{Thm} \label{mainthm}
Let $\pi \colon X \to \mbP^1$ be a del Pezzo fibration of degree $1$ with the following properties:
\begin{enumerate}
\item $X$ has only terminal quotient singular points of type $\frac{1}{2} (1,1,1)$.
\item Every fiber $F_t$ of $\pi$ over $t \in \mbP^1$ is embedded in $\mbP (1,1,2,3)$ as a weighted hypersurface of degree $6$ in such a way that $\mcO_{F_t} (-K_X|_{F_t})$ is isomorphic to $\mcO_{F_t} (1)$.
\end{enumerate}
If $X/\mbP^1$ satisfies both $K$-condition and $K^3_{3/2}$-condition, then $X/\mbP^1$ is birationally superrigid and, in particular, $X$ is not rational.
\end{Thm}

As an application we investigate birational superrigidity of del Pezzo fibrations of degree $1$ of typical type. 
The following supports Conjecture \ref{conj}, where we refer readers to Section \ref{sec:toricbundles} for the definition of toric $\mbP (1,1,2,3)$-bundles over $\mbP^1$.

\begin{Thm} \label{mainthm2}
Let $X/\mbP^1$ be a singular del Pezzo fibration of degree $1$ with only $\frac{1}{2} (1,1,1)$ points embedded as a hypersurface in a toric $\mbP (1,1,2,3)$-bundle over $\mbP^1$.
If $X/\mbP^1$ satisfies the $K$-condition, then it satisfies the $K^3_1$-condition. 
In particular, $X/\mbP^1$ is birationally superrigid if and only if it satisfies the $K$-condition.
\end{Thm}

We believe that most of the del Pezzo fibrations of degree $1$ with at most $\frac{1}{2} (1,1,1)$ points can be embedded in a toric $\mbP (1,1,2,3)$-bundle over $\mbP^1$ as a hypersurface.
In fact this is true when $X$ is nonsingular.
We pose the following.

\begin{Question}
Is every del Pezzo fibration of degree $1$ with at most $\frac{1}{2} (1,1,1)$ singularities embedded as a hypersurface in a toric $\mbP (1,1,2,3)$-bundle over $\mbP^1$?
\end{Question}

If the answer is yes, then Theorem \ref{mainthm2} shows that Conjecture \ref{conj} is true for del Pezzo fibrations of degree $1$ with at most $\frac{1}{2} (1,1,1)$ singularities.

There are several versions of the definition of birational rigidity for Mori fiber spaces.
We compare these notions and explain subtle differences in Section \ref{sec:variousrig}, where we in particular give an example of a nonsingular del Pezzo fibration of degree $1$ which is ``birationally rigid" but does not satisfy the $K$-condition. 
This should not be regarded as a counterexample to Conjecture \ref{conj} because the fibration is not ``birationally rigid over the base".

\begin{Ack}
The author would like to thank Professor Ivan Cheltsov for useful comments on an earlier version of the paper.
He also would like to thank Doctors Hamid Ahmadinezhad and Igor Krylov for useful comments.
\end{Ack}

\section{Preliminaries}

\subsection{Definition of birational superrigidity}

\begin{Def}
Let $X$ be a normal projective $\mbQ$-factorial variety with only terminal singularities admitting a surjective morphism $\pi \colon X \to S$ to a normal projective variety $S$.
We say that $\pi \colon X \to S$ (or we often denote it by $X/S$) is a {\it Mori fiber space} if $\pi$ has connected fibers, $-K_X$ is $\pi$-ample, the relative Picard number of $\pi$ is $1$ and $\dim S < \dim X$.

A Mori fiber space $X/\mbP^1$ with $\dim X = 3$ is called a {\it del Pezzo fibration} over $\mbP^1$.
A general fiber $F$ of a del Pezzo fibration $X/\mbP^1$ is a nonsingular del Pezzo surface and the {\it degree} of $X/\mbP^1$ is defined to be the degree $K_F^2$ of a general fiber $F$. 
\end{Def}

\begin{Def}
A birational map $f \colon X \ratmap X'$ between Mori fiber spaces $X/S$ and $X'/S'$ is {\it square} if there is a birational map $g \colon S \ratmap S'$ such that the diagram
\[
\xymatrix{
X \ar[d] \ar@{-->}[r]^f & X' \ar[d] \\
S \ar@{-->}[r]_g & S'}
\]
commutes and the induced birational map between generic fibers of $X/S$ and $X'/S'$ is a biregular isomorphism.
\end{Def}

\begin{Def} \label{def:birsuperrig}
We say that a Mori fiber space $X/S$ is {\it birationally superrigid} if any birational map $f \colon X \ratmap X'$ to a Mori fiber space $X'/S'$ is square.
\end{Def} 

We refer readers to Section \ref{sec:variousrig} for various versions of birational rigidity and comparison between them.

\begin{Rem} \label{rem:defrig}
As explained in the introduction, Grinenko \cite{Gri00,Gri03,Gri06} showed that a nonsingular del Pezzo fibration of degree $1$ over $\mbP^1$ is birationally rigid if and only if it satisfies $K$-condition.
The definition of birational rigidity there is as follows: a del Pezzo fibration $X/\mbP^1$ is {\it birationally rigid} if any birational map $\varphi \colon X \ratmap X'$ to a Mori fiber space $X'/S'$ is birational over the base, that is, there is an isomorphism $\mbP^1 \to S'$ that makes the diagram
\[
\xymatrix{
X \ar[d] \ar@{-->}[r]^{\varphi} & X' \ar[d] \\
\mbP^1 \ar[r]^{\cong} & S'}
\]
commutative.
Here we emphasize that $\varphi$ is not assumed to induce biregular automorphism between generic fibers of $X/\mbP^1$ and $X'/S'$, that is, $\varphi$ is not necessarily square.
However, if $X/\mbP^1$ is of degree $1$, then a birational map $\varphi \colon X \ratmap X'$ over the base is square.
Thus birational rigidity in the sense of Grinenko is equivalent to birational superrigidity for del Pezzo fibrations of degree $1$.
\end{Rem}

\subsection{Framework of proof}

In this subsection, let $\pi \colon X \to \mbP^1$ be a del Pezzo fibration.
We do not impose any condition on the degree of the fibration or the singularities of $X$ unless otherwise specified.
We explain the fact, which is well known at least when $X$ is nonsingular, that the failure of birational rigidity for $X/\mbP^1$ implies the existence of a movable linear system $\mcH$ on $X$ which is very singular.

\begin{Def}
Let $V$ be a normal variety, $D$ a $\mbQ$-Cartier $\mbQ$-divisor on $V$ (not necessarily effective), $\mcH$ a movable linear system of $\mbQ$-Cartier divisors on $X$ and $r$ a rational number.
We say that the pair $(V, D + r \mcH)$ is {\it canonical} if for any exceptional divisor $E$ over $X$, we have
\[
a_E (K_V) \ge \mult_E (D) + r \mult_E (\mcH),
\]
where $a_E (K_V)$ denotes the discrepancy of $K_V$ along $E$.
\end{Def}

For a del Pezzo fibration $X/\mbP^1$, we see that $\Pic (X) \otimes \mbQ$ is generated by $-K_X$ and the fiber class $F$.
Hence, for any linear system $\mcH$ on $X$, we have $\mcH \sim_{\mbQ} - n K_X + m F$ for some $n, m \in \mbQ$, which means that any member of $\mcH$ is $\mbQ$-linearly equivalent to $- n K_X + m F$.

\begin{Def} \label{def:weakmaxsing}
Let $\mcH \sim_{\mbQ} - n K_X + m F$ a movable linear system on $X$ with $n > 0$.
A {\it maximal singularity} of $\mcH$ is an exceptional divisor $E$ over $X$ for which 
\[
\mult_E (\mcH) > n a_E (K_X).
\]
The center of $E$ on $X$ is called a {\it maximal center} of $\mcH$.

We say that an irreducible subvariety $\Gamma \subset X$ is a {\it maximal center} if it is a maximal center of some movable linear system $\mcH \sim_{\mbQ} - n K_X + m F$ with $n > 0$.
\end{Def}

Suppose we are given a birational map $f \colon X \ratmap X'$ to a Mori fiber space $X'/S'$.
Let $\mcH' = \left| - n' K_{X'} + {\pi'}^* A' \right|$ be a very ample complete linear system on $X'$, where $n'$ is a positive integer and $A'$ is an ample divisor on $S'$.
The birational transform $\mcH$ of $\mcH'$ via $f$ is called a {\it movable linear system associated to $f$}.
A priori, we have $\mcH \sim_{\mbQ} - n K_X + m F$ for some $m, n \in \mbQ$.
However, after replacing $\mcH'$ with $\left| l (-n' K_{X'} + {\pi'}^*A') \right|$ for a divisible $l > 0$, we may assume that $\mcH \sim - n K_X + m F$, that is, $\mcH \subset \left| - n K_X + m F \right|$, with $m, n \in \mbZ$. 

An irreducible curve on $X$ is called {\it vertical} if it is contained in a $\pi$-fiber, otherwise it is called {\it horizontal}.

The following result is a direct and a slight generalization of Corti's argument in \cite{Corti} (and is also a generalization of the original argument given by Pukhlikov \cite{Puk}), which explains a framework of proof of birational rigidity.

\begin{Prop} \label{prop:birrigbasic}
Let $\pi \colon X \to \mbP^1$ be a del Pezzo fibration.
Suppose that we are given a non-square birational map $f \colon X \ratmap X'$ to a Mori fiber space $\pi' \colon X' \to S'$ and let $\mcH \subset \left|- n K_X + m F \right|$ be a movable linear system associated to $f$.
Suppose in addition that no horizontal curve on $X$ is a maximal center of $\mcH$.
If $m \ge 0$, then there exist irreducible subvarieties $\Gamma_1,\dots,\Gamma_k$ of $X$ contained in distinct $\pi$-fibers and positive rational numbers $\lambda_1,\dots,\lambda_k$ with the following properties:
\begin{enumerate}
\item $\Gamma_i$ is either a point or a curve.
\item $(X, - \sum \lambda_i F_i + \frac{1}{n} \mcH)$ is not canonical along $\Gamma_1,\dots,\Gamma_k$, where $F_i$ is the $\pi$-fiber containing $\Gamma_i$.
\item $\sum_{i=1}^k \lambda_i > m/n$.
\end{enumerate}
\end{Prop}

\begin{proof}
Let $\mcH' = \left| - n' K_{X'} + {\pi'}^*A' \right|$ be a very ample complete linear system on $X'$ whose birational transform via $f$ is $\mcH$.
By the Noether--Fano--Iskovskikh inequalities \cite[Theorem 2.4]{Corti}, we have $n > n'$ and $K_X + \frac{1}{n} \mcH$ is not canonical.
Let $p \colon W \to X$ and $q \colon W \to X'$ be a resolution of the indeterminacy of $f$.
By assumption, a maximal center of $\mcH$ is contained in a $\pi$-fiber and let $F_1,\dots,F_k$ be the $\pi$-fibers containing the maximal centers of $\mcH$.
For $i = 1,\dots,k$, Let $E_{ij}$ be the prime $p$-exceptional divisors whose center on $X$ is contained in $F_i$ and let $\{G_l\}$ be the other prime $p$-exceptional divisors.
Let $\mcH_W$ and $F'_i$ be the strict transforms of $\mcH$ and $F_i$ via $p$, respectively.

We can write
\[
\begin{split}
K_W &= p^*K_X + \sum a_{i j} E_{ij} + \sum a_l G_l, \\
p^*\mcH &= \mcH_W + \sum m_{i j} E_{ij} + \sum m_l G_l, \\
F'_i &= p^*F_i - \sum c_{i j} E_{ij}.
\end{split}
\]
Since $\mcH$ has a maximal center at some point in $F_i$, there is $j$ such that $m_{i j} > n a_{i j}$. 
We define 
\[
\lambda_i :=
\max_j \left\{ \frac{m_{i j} - n a_{i j}}{n c_{i j}} \right\} > 0.
\]
Let $j (i)$ be such that $\lambda_i = (m_{i j(i)} - n a_{i j(i)})/n c_{i j(i)}$.
The center $\Gamma_i := p (E_{i j (i)}) \subset F_i$ is necessarily a maximal center of $\mcH$. 
Then we have
\[
K_W - \sum \lambda_i F'_i + \frac{1}{n} \mcH_W
= p^* \left( K_X - \sum \lambda_i F_i + \frac{1}{n} \mcH \right) + \sum \alpha_{ij} E_{ij} + G, 
\]
where $G = \sum (a_l - \frac{1}{n} m_l) G_l$ and
\[
\alpha_{ij} = a_{ij} + \lambda_i c_{ij} - \frac{1}{n} m_{ij} \ge 0.
\]
Note that $\alpha_{i j (i)} = 0$, so that $K_X - \sum \lambda_i F_i + \frac{1}{n} \mcH$ is strictly canonical.
Note also that $G$ is effective.

We set $\eta = \frac{m}{n} - \sum \lambda_i$ and assume that $\eta \ge 0$.
Then, since 
\[
K_X - \sum \lambda_i F_i + \frac{1}{n} \mcH \sim_{\mbQ} \eta F, 
\]
we see that the divisor
\[
K_W + \frac{1}{n} \mcH_W \sim_{\mbQ} \eta p^* F + \sum \alpha_{ij} E_{ij} + G + \sum \lambda_i F'_i
\]
is effective.
It then follows that
\[
q_* \left(K_W + \frac{1}{n} \mcH_W \right)
= K_{X'} + \frac{1}{n} \mcH'
= \left(K_{X'} + \frac{1}{n'} \mcH' \right) - \left(\frac{1}{n'} - \frac{1}{n} \right) \mcH'
\]
is effective.
This implies 
\[
\frac{1}{n'} - \frac{1}{n} \le 0.
\] 
This is a contradiction since $n > n'$, and the inequality $\sum \lambda_i > m/n$ is proved.
Now we replace $\lambda_i$ with $\lambda_i - \varepsilon$ for a sufficiently small $\varepsilon > 0$.
Then it is easy to observe that the assertions (i), (ii) and (iii) are satisfied.
\end{proof}

The assumption on horizontal curves is always satisfied for del Pezzo fibrations of degree $1$.

\begin{Lem} \label{lem:exclhorcurve}
Let $\pi \colon X \to \mbP^1$ be a del Pezzo fibration of degree $1$.
Then no horizontal curve is a maximal center.
\end{Lem}

\begin{proof}
See \cite[\S 3]{Puk}.
Although $X$ is assume to be nonsingular in \cite{Puk}, the same proof works for an arbitrary del Pezzo fibration of degee $1$.
\end{proof}

\subsection{An excluding method for singular points}

For a $3$-dimensional terminal quotient singular point $\msp \in U$ of type $\frac{1}{r} (1,a,r-a)$, where $0 < a < r$ and $\gcd (r,a) = 1$, the weighted blowup $\varphi \colon V \to U$ at $\msp$ with weight $\frac{1}{r} (1,a,r-a)$ (which is called as the {\it Kawamata blowup}) is the unique divisorial contraction centered at $\msp$ (see \cite{Kawamata}).
Note that, for the $\varphi$-exceptional divisor $E$, we have $E \cong \mbP (1,a,r-a)$, $K_V = \varphi^* K_U + \frac{1}{r} E$ and
\[
(E^3) = \frac{r^2}{a (r-a)}.
\] 

We give an excluding criterion which is a generalization of \cite[Lemma 3.2.8]{CP} to del Pezzo fibrations.

\begin{Lem} \label{lem:exclcriquot}
Let $\pi \colon X \to \mbP^1$ be a del Pezzo fibration.
Let $\msp \in X$ be a terminal quotient singular point and let $\varphi \colon Y \to X$ the Kawamata blowup of $X$ at $\msp$ with exceptional divisor $E$.
Suppose that there are infinitely many distinct irreducible curves $C_{\lambda}$, $\lambda \in \Lambda$, on $Y$ such that $(-K_Y \cdot C_{\lambda}) \le 0$, $(E \cdot C_{\lambda}) > 0$ and $C_{\lambda}$ is mapped to a point by $\pi \circ \varphi$ for any $\lambda \in \Lambda$.
Then $\msp$ is not a maximal center.
\end{Lem}

\begin{proof}
Suppose that $\msp \in X$ is a maximal center.
Then there exists a movable linear system $\mcH \sim_{\mbQ} - n K_X + m F$, where $n > 0$, on $X$ such that $(X, \frac{1}{n} \mcH)$ is not canonical at $\msp$.
By \cite[Lemma 7]{Kawamata}, $E$ is a maximal singularity (of $\mcH$).
Hence we have
\[
K_Y + \frac{1}{n} \tilde{\mcH} \sim_{\mbQ} \varphi^* \left(K_X + \frac{1}{n} \mcH \right) - c E
\sim_{\mbQ} \frac{m}{n} \varphi^*F - c E
\]
for some $c > 0$, where $\tilde{\mcH}$ is the birational transform of $\mcH$ via $\varphi$.
Since $(-K_Y \cdot C_{\lambda}) \le 0$, $(E \cdot C_{\lambda}) > 0$ and $(\varphi^*F \cdot C_{\lambda}) = 0$, we have
\[
(\tilde{\mcH} \cdot C_{\lambda}) = n (-K_Y \cdot C_{\lambda}) +  m (\varphi^*F \cdot C_{\lambda}) - c n (E \cdot C_{\lambda}) < 0.
\]
This implies that $C_{\lambda}$ is contained in the base locus of $\tilde{\mcH}$.
This is a contradiction since $\tilde{\mcH}$ is movable.
\end{proof}

\subsection{A result on some weighted hypersurfaces}

Let $\mbP (1,1,2,3)$ be the weighted projective space, defined over $\mbC$, with homogeneous coordinates $x,y,z,w$ of weight $1,1,2,3$, respectively. 
A curve $\ell$ on $\mbP (1,1,2,3)$ defined by $\alpha x + \beta y = z + q (x,y) = 0$ for some $\alpha, \beta \in \mbC$ with $(\alpha,\beta) \ne (0,0)$ and a quadric $q (x,y)$ is called a $\frac{1}{3}$-{\it line}.
For a degree $6$ weighted hypersurface $S$ in $\mbP (1,1,2,3)$, $S$ does not contain any $\frac{1}{3}$-line if and only if $w^2$ appears in its defining polynomial with non-zero coefficient.

\begin{Lem} \label{lem:wsurf}
Let $S$ be an irreducible and reduced weighted hypersurface of degree $6$ in $\mbP (1,1,2,3)$ which does not contain a $\frac{1}{3}$-line.
\begin{enumerate}
\item For any $1$-cycle $\Delta$ on $S$, there is a curve $C \in |\mcO_S (1)|$ such that $C \cap \operatorname{Supp} (\Delta) \ne \emptyset$ but $C$ does not contain any component of $\operatorname{Supp} (\Delta)$.
\item For any point $\msp \in S$, the linear system $|\mcI_{\msp} \mcO_S (2)| \subset |\mcO_S (2)|$ of curves passing through $\msp$ is movable.
\end{enumerate}
\end{Lem}

\begin{proof}
We see that $(x = y = 0) \cap S$ consists of a single point since $w^2 \in S$.
Thus (i) follows by taking $C$ as a general curve in $|\mcO_S (1)|$.
We prove (ii).
If $\msp \in (x = y = 0) \cap S$, then it is clear that $\Bs |\mcI_{\msp} \mcO_S (2)| = \{\msp\}$.
Suppose that $\msp \notin (x = y = 0)$.
By a coordinate change, we may assume $\msp = (1\!:\!0\!:\!0\!:\!0)$.
Then $\Bs |\mcI_{\msp} \mcO_S (2)| = (y = z = 0) \cap S$ set-theoretically.
Since $(y = z = 0) \subset \mbP (1,1,2,3)$ is a $\frac{1}{3}$-line which is not contained in $S$, we conclude that $\Bs |\mcI_{\msp} \mcO_S (2)|$ is a finite set of points (including $\msp$).
This completes the proof.
\end{proof}

\section{Birational superrigidity of del Pezzo fibration of degree $1$}

Throughout the present section, let $\pi \colon X \to \mbP^1$ be a del Pezzo fibration of degree $1$ satisfying the conditions (i) and (ii) of Theorem \ref{mainthm}.
Neither $K$-condition nor $K^3_{3/2}$-condition is assumed unless explicitly stated.

\begin{Rem}
Let $F_o$, $o \in \mbP^1$, be a fiber of $\pi$.
Then we have the following properties:
\begin{enumerate}
\item $F_o$ is irreducible and reduced.
\item $F_o$ does not contain a $\frac{1}{3}$-line.
\end{enumerate}
It is clear that $F_o$ is irreducible since the Picard number of $X$ is $2$.
By the classification result \cite{MP} of multiple fibers of del Pezzo fibrations, a del Pezzo fibrations of degree $1$ with at most $\frac{1}{2} (1,1,1)$ singular points cannot have a multiple fiber.
This shows (i).
If $F_o$ contains a $\frac{1}{3}$-line $\ell$, then we have 
\[
(-2 K_X \cdot \ell) = (-2 K_X|_{F_o} \cdot \ell) = (\mcO_{F_o} (2) \cdot \ell) = 2/3.
\]
But this cannot happen since $-2 K_X$ is a Cartier divisor and thus $(-K_X \cdot \ell)$ is an integer.
This shows (ii).
\end{Rem}

\begin{Rem} \label{rem:liftdivFX}
Let $F_o$ be a fiber, $o \in \mbP^1$, and $C \in |\mcO_{F_o} (m)|$ a divisor, where $m$ is a positive integer.
Then there exists an integer $n$ and $D \in |- m K_X + n F|$ such that $D|_{F_o} = C$.

Indeed, the restriction sequence 
\[
0 \to \mcO_X (- m K_X + (n-1)F) \to \mcO_X (- m K_X + n F) \to \mcO_{F_o} (m) \to 0
\]
is exact (cf.\ \cite[Proposition 5.26]{KM}).
We see that $- (m+1) K_X + (n-1)F$ is ample for a sufficiently large $n$ so that $H^1 (X, \mcO_X (- m K_X + (n-1)F)) = 0$ by Kawamata-Viehweg vanishing theorem and Serre duality.
It follows that the map 
\[
H^0 (X, \mcO_X (-m K_X + n F)) \to H^0 (F_o, \mcO_{F_o} (m))
\] 
is surjective and the existence of $D$ follows.
\end{Rem}

\begin{Prop} \label{prop:exclsingpt}
No singular point of type $\frac{1}{2} (1,1,1)$ on $X$ is a maximal center.
\end{Prop}

\begin{proof}
Let $F_o$ be the $\pi$-fiber containing $\msp$.
By Lemma \ref{lem:wsurf}.(ii), there are infinitely many curves $C_{\lambda} \in |\mcO_{F_o} (2)|$ passing through $\msp$.
By Remark \ref{rem:liftdivFX}, for each $\lambda$, there exists $D_{\lambda} \in |- 2 K_X + n F|$ such that $D_{\lambda}|_{F_o} = C_{\lambda}$.
Let $\varphi \colon Y \to X$ be the Kawamata blowup at $\msp$ with exceptional divisor $E \cong \mbP^2$.
We write $\varphi^* D_{\lambda} = \tilde{D}_{\lambda} + e E$ and $\varphi^*F_o = \tilde{F}_o + e' E$, where $\tilde{D}_{\lambda}$ and $\tilde{F}_o$ are proper transform on $Y$ of $D_{\lambda}$ and $F_o$, respectively.
We have $e, e' \ge 1$ since $D_{\lambda}, F_o \ni \msp$ are Cartier divisors on $X$.
We have
\[
(-K_Y \cdot \tilde{D}_{\lambda} \cdot \tilde{F}_o) = (-K_X \cdot D_{\lambda} \cdot F_o) - \frac{e e'}{2} (E^3) = 2 - 2 e e' \le 0.
\]
Denote by $\tilde{C}_{\lambda}$ the proper transform of the curve $C_{\lambda}$ on $Y$.
If we write $\tilde{D}_{\lambda} \cdot \tilde{F}_o = \tilde{C}_{\lambda} + \Delta$ for some effective $1$-cycle $\Delta$ on $Y$, then the support of $\Delta$ is contained in $E$ and, in particular, $(-K_Y \cdot \Delta) \ge 0$.
Thus
\[
(-K_Y \cdot \tilde{C}_{\lambda}) = (-K_Y \cdot \tilde{D}_{\lambda} \cdot \tilde{F}_o) - (-K_Y \cdot \Delta) \le 0.
\]
Hence there exists a component $\tilde{C}_{\lambda}^{\circ}$ of $\tilde{C}_{\lambda}$ such that $(-K_Y \cdot \tilde{C}_{\lambda}^{\circ}) \le 0$.
The curve $\tilde{C}_{\lambda}^{\circ}$ is a proper transform of $C_{\lambda}^{\circ} = \varphi_* \tilde{C}_{\lambda}^{\circ}$ and necessarily satisfies $(E \cdot \tilde{C}_{\lambda}^{\circ}) > 0$.

Therefore there exist infinitely many curves $\tilde{C}_{\lambda}^{\circ}$ on $Y$ which intersect $-K_Y$ non-positively and $E$ positively.
By Lemma \ref{lem:exclcriquot}, $\msp$ is not a maximal center.
\end{proof}

\begin{Prop} \label{prop:exclcurve}
No curve on $X$ is a maximal center of $\mcH$. 
\end{Prop}

\begin{proof}
Let $\Gamma \subset X$ be an irreducible and reduced curve and assume that $\Gamma$ is a maximal center of a movable linear system $\mcH \subset \left|- n K_X + m F \right|$.
If $\Gamma$ passes through a $\frac{1}{2} (1,1,1)$ point, then, by \cite[Lemma 7]{Kawamata} (see also \cite[Theorem 2.2.1]{CP}), the $\frac{1}{2} (1,1,1)$ point is a maximal center of $\mcH$, which is impossible by Proposition \ref{prop:exclsingpt}.
Thus it is enough to derive a contradiction assuming that $\Gamma$ is contained in the nonsingular locus of $X$.
Note that we have $\mult_{\Gamma} (\mcH) > n$.
The rest of the proof is the same as the one given by Pukhlikov in \cite{Puk}.
We include it for readers' convenience.

By Lemma \ref{lem:exclhorcurve}, we may assume that $\Gamma$ is vertical and let $F_o$, $o \in \mbP^1$, be the $\pi$-fiber containing $\Gamma$.
By Lemma \ref{lem:wsurf}.(i), there is a curve $C \in |\mcO_{F_o} (1)|$ which intersects $\Gamma$ and but does not contain $\Gamma$.
Then, by Remark \ref{rem:liftdivFX}, we can take $D \in \left| - K_X + l F \right|$ (for some large $l > 0$) such that $C = D \cdot F_o$.
Take a point $\msp \in \Gamma \cap C$ and let $H \in \mcH$ be a general member.
We have 
\[
n = (H \cdot D \cdot F_o) = (H \cdot C) \ge \mult_{\msp} (H) \ge \mult_{\Gamma} (H) > n.
\]
This is a contradiction.
Therefore $\Gamma$ is not a maximal center.
\end{proof}

\begin{Rem} \label{rem:nefthreshold}
We explain a formulation of the nef threshold $\nef (X/\mbP^1)$ in terms of intersection number.
We denote by $R \subset \bNE (X)$ the extremal ray which is not generated by a curve contracted by $\pi$.
Let $\xi \in R$ be a class.
Then we have $(-K_X + \nef (V/\mbP^1) F \cdot \xi) = 0$ and $(F \cdot \xi) > 0$.
It follows that
\[
\nef (X/\mbP^1) = - \frac{(-K_X \cdot \xi)}{(F \cdot \xi)}.
\]
\end{Rem}

\begin{Prop} \label{prop:square}
Suppose that we are given a birational map $f \colon X \ratmap X'$ to a Mori fiber space $\pi' \colon X' \to S'$ and let $\mcH \subset \left|- n K_X + m F \right|$ be a movable linear system associated to $f$.
If $m \ge 0$ and $X/\mbP^1$ satisfies $K^3_{3/2}$-condition, then $f$ is square.
\end{Prop}

\begin{proof}
We assume that $f$ is not square.
In view of Lemma \ref{lem:exclhorcurve}, Propositions \ref{prop:exclsingpt}, \ref{prop:exclcurve} and the assumption $m \ge 0$, Proposition \ref{prop:birrigbasic} implies the existence of nonsingular points $\msp_1,\dots,\msp_k \in X$ and positive rational numbers $\lambda_1,\dots,\lambda_k$ such that $(X, - \sum \lambda_i F_i + \frac{1}{n} \mcH)$ is not canonical at each $\msp_i$ and $\sum \lambda_i > m/n$, where $F_i$ is the $\pi$-fiber containing $\msp_i$.

Let $H_1, H_2 \in \mcH$ be general members and set $Z = H_1 \cdot H_2$ which is an effective $1$-cycle on $X$.
We decompose $Z = Z_{\mathrm{h}} + Z_{\mathrm{v}}$ into the horizontal component $Z_{\mathrm{h}}$ and the vertical component $Z_{\mathrm{v}}$.
We denote by $Z_{F_i}$ the component of $Z_{\mathrm{v}}$ whose support is contained in the fiber $F_i$, so that $Z_{\mathrm{v}} = \sum Z_{F_i}$. 

By the Corti inequality \cite[Theorem 3.12, Remark 3.13]{Corti}, for $i = 1,\dots,k$, there is a number $t_i$ with $0 < t_i \le 1$ such that
\[
\mult_{\msp_i} Z_{\mathrm{h}} + t_i \mult_{\msp_i} Z_{F_i} \ge 4 (1 + \lambda_i t_i) n^2.
\]
We have 
\[
\mult_{\msp_i} Z_{\mathrm{h}} \le (F_i \cdot Z_{\mathrm{h}}) = (F \cdot Z_{\mathrm{h}}) = (F \cdot Z) = (F \cdot H_1 \cdot H_2) = n^2.
\]
It follows that
\[
\mult_{\msp_i} Z_{F_i} \ge \frac{3}{t_i} n^2 + 4 \lambda_i n^2 \ge 3 n^2 + 4 \lambda_i n^2.
\]
By Lemma \ref{lem:wsurf}, there exists a curve $C_i \in |\mcO_{F_i} (2)|$ which passes through $\msp_i$ and which does not contain any component of $\operatorname{Supp} (Z_{F_i})$.
By Remark \ref{rem:liftdivFX}, there is a divisor $D_i \in |- 2 K_X + l F|$ such that $D_i|_{F_i} = C_i$.
Then we have
\[
(-K_X \cdot Z_{F_i}) = \frac{1}{2} (D_i \cdot Z_{F_i})  
\ge \frac{1}{2} \mult_{\msp_i} Z_{F_i} \ge \frac{3}{2} n^2 + 2 \lambda_i n^2
\]
and, by taking into account the inequality $\sum \lambda_i > m/n$, we have
\begin{equation} \label{eq:nsptdegineq1}
(-K_X \cdot Z_{\mathrm{v}}) = \sum_{i=1}^k (-K_X \cdot Z_{F_i})
\ge \frac{3}{2} k  n^2 + 2 n^2 \sum_{i=1}^k \lambda_i
> \frac{3}{2} n^2 + 2 m n.
\end{equation}

Let $\ell \in \bNE (X)$ be the class such that $\ell = \mcO_F (1)$ so that $\mbR_{\ge 0} \cdot \ell$ is the extremal ray corresponding to $\pi$.
Let $\xi \in \bNE (X)$ be the class generating the other extremal ray of $\bNE (X)$.
We have $(F \cdot \xi) > 0$.
By multiplying a suitable positive rational number, we assume that $(F \cdot \xi) = 1$.
We write $[Z_{\mathrm{h}}] = \alpha \xi + \beta \ell$, where $[Z_{\mathrm{h}}]$ denotes the class in $\bNE (X)$.
Note that $\alpha, \beta \ge 0$.
Since $\alpha = (F \cdot Z_{\mathrm{h}}) = n^2$ and $(-K_X \cdot \xi) = - \nef (X/\mbP^1)$ by Remark \ref{rem:nefthreshold}, we have
\[
(-K_X \cdot Z_{\mathrm{h}}) = n^2 (-K_X \cdot \xi) + \beta = -n^2 \nef (X/\mbP^1) + \beta \ge - n^2 \nef (X/\mbP^1).
\]
Moreover we have
\[
(-K_X \cdot Z) = (-K_X) \cdot (- n K_X + m F)^2 = n^2 (-K_X)^3 + 2 m n.
\]
Thus,
\begin{equation} \label{eq:nsptdegineq2}
\begin{split}
(-K_X \cdot Z_{\mathrm{v}}) &= (-K_X \cdot Z) - (-K_X \cdot Z_{\mathrm{h}}) \\
&\le n^2 (-K_X)^3 + 2 m n + n^2 \nef (X/\mbP^1) \\
&\le \frac{3}{2} n^2 + 2 m n,
\end{split}
\end{equation}
where the last inequality follows from $K^3_{3/2}$-condition.
The inequalities \eqref{eq:nsptdegineq1} and \eqref{eq:nsptdegineq2} are impossible and the proof is completed.
\end{proof}

\begin{proof}[of \emph{Theorem \ref{mainthm}}]
Let $X/\mbP^1$ be a del Pezzo fibration of degree $1$ as in Theorem \ref{mainthm}.
Suppose we are given a birational map $f \colon X \ratmap X'$ to a Mori fiber space $X'/S'$ and let $\mcH \sim_{\mbQ} - n K_X + m F$ be the birational transform of a very ample complete linear system $\mcH'$ on $X'$.
By the $K$-condition, we have $m \ge 0$.
Thus, by Proposition \ref{prop:square}, $f$ is square.
Therefore $X/\mbP^1$ is birational superrigid.
\end{proof}

\section{del Pezzo fibrations embedded as a hypersurface in a toric $\mbP (1,1,2,3)$-bundles over $\mbP^1$}
\label{sec:birrigdP}

\subsection{Toric $\mbP (1,1,2,3)$-bundles over $\mbP^1$}
\label{sec:toricbundles}

We construct del Pezzo fibrations $X/\mbP^1$ as hypersurfaces in suitable toric $\mbP (1,1,2,3)$-bundles over $\mbP^1$.
We refer readers to \cite{CLS} for Cox rings of toric varieties.

For $\lambda,\mu,\nu \in \mbZ$ with $\lambda \ge 0$, let $P = P (\lambda,\mu,\nu)$ be the projective simplicial toric variety with Cox ring 
\[
\Cox (P) = \mbC [u,v,x,y,z,w]
\] 
which is $\mbZ^2$-graded as
\[
\begin{pmatrix}
u & v & x & y & z & w \\
1 & 1 & 0 & \lambda & \mu & \nu \\
0 & 0 & 1 & 1 & 2 & 3
\end{pmatrix}
\]
and with the irrelevant ideal $I = (u,v) \cap (x,y,z,w)$, that is, $P$ is the geometric quotient
\[
P = (\mbA^6 \setminus V (I)) / (\mbC^*)^2,
\]
where the action of $(\mbC^*)^2$ on $\mbA^6 = \Spec \Cox (P)$ is given by the above $2 \times 6$ matrix.

\begin{Rem} \label{rem:toricbundle}
A priori, as a toric $\mbP (1,1,2,3)$-bundle $P$ over $\mbP^1$, we need to consider $P$ given by an action of the form
\[
\begin{pmatrix}
u & v & x & y & z & w \\
1 & 1 & \alpha & \beta & \gamma & \delta \\
0 & 0 & 1 & 1 & 2 & 3
\end{pmatrix},
\]
where $\alpha,\beta,\gamma,\delta \in \mbZ$.
Considering a suitable $\SL_2 (\mbZ)$-action on the matrix, we see that $P = P (\lambda,\mu,\nu)$ for some $\lambda,\mu,\nu \in \mbZ$.
Moreover, possibly interchanging $x$ and $y$, we may assume that $\lambda \ge 0$.
\end{Rem}

Let $P = P (\lambda,\mu,\nu)$ be as above.
The Weil divisor class group $\Cl (P)$ is isomorphic to $\mbZ^2$.
There is a natural morphism $\Pi \colon P \to \mbP^1$ defined as the projection to the coordinates $u,v$, and this realizes $P$ as a weighted projective space bundle over $\mbP^1$ whose fibers are $\mbP (1,1,2,3)$.

Let $F$ and $H$ be the Weil divisor classs corresponding to ${}^t (1, 0)$ and ${}^t (0,1)$, respectively.
Note that $F$ is the $\Pi$-fiber class and the restriction of $H$ to a fiber corresponds to $\mcO_{\mbP (1,1,2,3)} (1)$.
For a coordinate $t \in \{u,v,x,y,z,w\}$, we denote by $D_t$ the zero locus of $t$ which is a torus invariant Weil divisor.
Note that 
\[
D_u, D_v \sim F,
\] 
\[
D_x \sim H, \ 
D_y \sim H + \lambda F, \ 
D_z \sim 2 H + \mu F, \ 
D_w \sim 3 H + \nu F.
\]
There are exactly four irreducible and reduced torus invariant curves, denoted by $C_x, C_y, C_z, C_w$, which are not contained in a $\Pi$-fiber.
They are intersections of three distinct divisors from $\{D_x, D_y, D_z, D_w\}$, that is, 
\[
C_z = D_x \cdot D_y \cdot D_w, \ 
C_w = D_x \cdot D_y \cdot D_z, 
\]
and similarly for $C_x$ and $C_y$.

It is easy to see that $\Sing (P) = C_z \cup C_w$ and singularities of $P$ along $C_z$ and $C_w$ are of type $\frac{1}{2} (1,1,1) \times \mbP^1$ and $\frac{1}{3} (1,1,2) \times \mbP^1$, respectively.

%

\begin{Lem} \label{lem:intnumP}
For the intersection numbers on $P$, we have
\[
(H^4) = - \frac{6 \lambda + 3 \mu + 2 \nu}{36}, \ 
(H^3 \cdot F) = \frac{1}{6}, \ 
(H^2 \cdot F^2) = (H \cdot F^3) = (F^4) = 0.
\]
\end{Lem}

\begin{proof}
The assertion $(H^2 \cdot F^2) = (H \cdot F^3) = (F^4) = 0$ is obvious since $F$ is the fiber class of $\Pi \colon P \to \mbP^1$.
Since $F \cong \mbP (1,1,2,3)$ and $\mcO_F (H|_F) \cong \mcO_F (1)$, we have
\[
(H^3 \cdot F) = (H|_F^3)_F = \frac{1}{6}.
\]
Since $D_x \cap D_y \cap D_z \cap D_w = \emptyset$, we have
\[
\begin{split}
0 &= (D_x \cdot D_y \cdot D_z \cdot D_w) = (H \cdot H + \lambda F \cdot 2 H + \mu F \cdot 3 H + \nu F) \\
&= 6 (H^4) + (6 \lambda + 3 \mu + 2 \nu) (H^3 \cdot F) = 6 (H^4) + \frac{6 \lambda + 3 \mu + 2 \nu}{6},
\end{split}
\] 
which completes the proof.
\end{proof}

\subsection{del Pezzo fibrations of degree $1$ with $\frac{1}{2} (1,1,1)$ points not satisfying the $K^2$-condition}

By a {\it del Pezzo fibration $\pi \colon X \to \mbP^1$ of degree $1$ embedded in a toric $\mbP (1,1,2,3)$-bundle as a hypersurface}, we mean a hypersurface $X$ of $P (\lambda,\mu,\nu)$ for some $\lambda,\mu,\nu$ such that $\pi = \Pi|_X$.

The aim of this section is to prove the following.

\begin{Thm} \label{thm:embdP}
Let $X/\mbP^1$ be a singular del Pezzo fibration of degree $1$ with only singular points of type $\frac{1}{2} (1,1,1)$ embedded in a toric $\mbP (1,1,2,3)$-bundle over $\mbP^1$ as a hypersurface.
If $X/\mbP^1$ satisfies the $K$-condition, then it satisfies the $K^3_1$-condition.
\end{Thm}

In the following let $\pi \colon X \to \mbP^1$ be a del Pezzo fibration of degree $1$ embedded in $P = P (\lambda,\mu,\nu)$ as a hypersurface.
We assume that $X$ is singular but has only singular points of type $\frac{1}{2} (1,1,1)$.
Note that we do not assume that $X/\mbP^1$ satisfies the $K$-condition.
As explained in Remark \ref{rem:toricbundle}, we may and do assume $\lambda \ge 0$ without loss of generality.

We fix notations which will be used in the rest of this section.
\begin{itemize}
\item $f = f (u,v,x,y,z,w) \in \mbC [u,v,x,y,z,w]$ is the defining polynomial of $X$ in $P$.
\item $F_X = F|_X$ and $H_X = H|_X$.
\item $N^1 (X) = (\operatorname{Div} (X)/\equiv) \otimes \mbR$ and $N^1 (P) = (\operatorname{Div} (P)/\equiv) \otimes \mbR$, where $\equiv$ denotes numerical equivalence on divisors.
\item $Q$ is the ray in $N^1 (X)$ spanned by $F_X$ and, for $t \in \{x,y,z,w\}$, $R_t$ is the ray in $N^1 (X)$ spanned by $D_t|_X$.
\item $\delta_X := (-K_X)^3 + \nef (X/\mbP^1)$.
\end{itemize}

\begin{Lem} \label{lem:basiccond}
$X$ is a member of $|6 H + 2 \nu F|$ and we have $\nu \ge 0$ and $3 \mu < 2 \nu$.
\end{Lem}

\begin{proof}
We see that $w^2$ appears in $f$ with nonzero constant coefficient since $X$ has only $\frac{1}{2} (1,1,1)$ singular points.
Hence $X \in |6 H + 2 \nu F|$ and, after replacing the coordinate $w$, we may write
\[
f = f_1 (u,v,x,y) + f_2 (u,v,x,y) z + f_3 (u,v,x,y) z^2 + a (u,v) z^3 + w^2
\]
for some $f_1, f_2, f_3 \in \mbC [u,v,x,y]$ and $a \in \mbC [u,v]$.

Suppose that $\nu < 0$.
Then, since $\lambda \ge 0$, $f$ consists of the terms divisible by $z$ or $w$, that is, $f_1 = 0$ as a polynomial.
This implies that $X$ contains the surface $S = D_z \cdot D_w$ which is a $\mbP^1$-bundle over $\mbP^1$ (i.e.\ a Hirzebruch surface).
We see that $\Gamma := (f_2 = 0) \cap S \subset X$ is a curve. 
Since $\Gamma$ is contained in the nonsingular locus of $P$ and $\mult_{\msp} X \ge 2$ for any $\msp \in \Gamma$, $X$ is singular along  $\Gamma$.
This is a contradiction and thus $\nu \ge 0$.

If $a (u,v) = 0$ as a polynomial, then $X$ contains the curve $C_z = D_x \cdot D_y \cdot D_w$.
This implies that $X$ is singular along $C_z$ and this is impossible.
Thus $a (u,v) \ne 0$ and this implies $3 \mu + \deg a = 2 \nu$.
Suppose that $a (u,v)$ is a non-zero constant.
Then $C_z \cap X = \emptyset$ and this implies that $X$ does not have a singular point of type $\frac{1}{2} (1,1,1)$.
This is impossible.
This shows that $\deg a > 0$ and thus $3 \mu < 2 \nu$.
\end{proof}

By adjunction, we have
\[
-K_X = -(K_P + X)|_X = H_X + (-\nu + \lambda + \mu + 2) F_X.
\] 

\begin{Lem} \label{lem:Kcube}
We have
\[
(-K_X)^3 = 2 \lambda + \frac{5}{2} \mu - 3 \nu + 6.
\]
\end{Lem}

\begin{proof}
This follows from a straightforward computation 
\[
(-K_X)^3 = (H + (- \nu + \lambda + \mu + 2)F)^3 \cdot (6 H + 2 \nu F)
\]
using the result of Lemma \ref{lem:intnumP}.
\end{proof}

We will classify the triplets $(\lambda,\mu,\nu)$ such that $X/\mbP^1$ does not satisfy the $K^2$-condition, or equivalently, such that $\delta_X > 0$ (cf.\ Remark \ref{rem:Kcondsequiv}).
In order to do so, we need to understand the nef cone $\Nef (X)$ which depends on the positions of the rays $R_x, R_y, R_z, R_w$ inside $N^1 (X)$.
We define {\it weight ratios} of the coordinates $x, y, z, w$ as follows:
\[
\wtr (x) := 0, \ 
\wtr (y) := \lambda, \ 
\wtr (z) := \mu/2, \ 
\wtr (w) := \nu/3.
\]
Note that $\wtr (x) \le \wtr (y)$, $\wtr (x) \le \wtr (w)$ and $\wtr (z) < \wtr (w)$ by Lemma \ref{lem:basiccond}.
The classification will be done by the following case division:
\begin{enumerate}
\item[(a)] $\wtr (y) \le \wtr (w)$. This case is further divided into the following: 
\begin{enumerate}
\item[(a-i)] $\max \{\wtr (x), \wtr (z)\} \le \wtr (y) \le \wtr (w)$.
\item[(a-ii)] $\wtr (x) \le \wtr (y) < \wtr (z) < \wtr (w)$.
\end{enumerate}
\item[(b)] $\wtr (w) < \wtr (y)$.
\end{enumerate}

\subsubsection{\emph{Case (a-i):} $\max \{\wtr (x), \wtr (z)\} \le \wtr (y) \le \wtr (w)$}

\begin{Lem} \label{lem:nef(a-i)}
\begin{enumerate}
\item $\Nef (X) = Q + R_y$ and $\nef (X/\mbP^1) = - \mu + \nu - 2$.
\item $\delta_X = 2 \lambda + \frac{3}{2} \mu - 2 \nu + 4$.
\end{enumerate}
\end{Lem}

\begin{proof}
We prove (i).
We have $2 D_y \sim 2 H + 2 \lambda F$ and $H^0 (P, \mcO_P (2 D_y))$ contains the sections
\[
x^2 u^{2 \lambda}, \ x^2 v^{2 \lambda}, \ y^2, \ z u^{2 \lambda - \mu}, \ z v^{2 \lambda - \mu}
\]
whose common zero locus is the curve $C_w \subset P$.
Since $X \cap C_w = \emptyset$ because of the presence of the term $w^2$ in the equation of $X$, we see that $2 D_y|_X$ is base point free.
This shows that $D_y|_X$ is nef.
We consider $\Gamma := D_x \cdot D_z \cdot X$, which is an effective $1$-cycle on $X$.
Since $X \sim 2 D_w$, we have
\[
(D_y|_X \cdot \Gamma) = 2 (D_x \cdot D_y \cdot D_z \cdot D_w) = 0.
\]
This shows that $D_y$ is not ample and thus lies in the boundary of $\Nef (X)$.
Now the computations of $\nef (X/\mbP^1)$ and $\delta_X$ are straightforward and we leave them to readers.
\end{proof}

\begin{Prop} \label{prop:case(a-i)}
Suppose that $X/\mbP^1$ does not satisfy the $K^2$-condition.
Then
\[
(\lambda,\mu,\nu) \in \{ (0,-2,0), (0,-1,0), (0,-1,1), (0,0,1), (1,1,3), (1,2,4), (2,3,6) \}.
\]
\end{Prop}

\begin{proof}
Since the $K^2$-condition is equivalent to the $K^3_0$-condition, it suffices to classify triplets $(\lambda,\mu,\nu)$ such that $\delta_X > 0$.
Since $\delta_X \in \frac{1}{2} \mbZ$, this is equivalent to $2 \delta_X \ge 1$ and hence we assume
\begin{equation} \label{eq:delta(a-i)}
4 \lambda + 3 \mu - 4 \nu + 8 \ge 1.
\end{equation}
By Lemma \ref{lem:basiccond} and the case assumptions, we have
\begin{equation} \label{eq:basecd(a-i)}
\mu \le 2 \lambda, \ 
0 \le 3 \lambda \le \nu, \ 
3 \mu \le 2 \nu - 1.
\end{equation}
By $3 \mu \le 2 \nu -1$ and \eqref{eq:delta(a-i)}, we have 
\begin{equation} \label{eq:(a-i)sub}
4 \lambda - 2 \nu + 6 \ge 0.
\end{equation}
Combining this with $3 \lambda \le \nu$, we have $\lambda \le 3$, that is,  $\lambda \in \{0,1,2,3\}$.

\begin{itemize}
\item If $\lambda = 0$, then we have $\nu \in \{0,1,2,3\}$ by $3 \mu \le 2 \nu -1$ and \eqref{eq:(a-i)sub}.
It is then easy to see that $(\lambda,\mu,\nu) \in \{ (0,-2,0), (0,-1,0), (0,-1,1) \}$.
\item If $\lambda = 1$, then we have $\nu \in \{3,4,5\}$ by $3 \lambda \le \nu$ and \eqref{eq:(a-i)sub}.
It is then easy to see that $(\lambda,\mu,\nu) \in \{ (1,1,3), (1,2,4) \}$.
\item If $\lambda = 2$, then $\nu \in \{6,7\}$ by $3 \lambda \le \nu$ and \eqref{eq:(a-i)sub}.
It is then easy to see that $(\lambda,\mu,\nu) = (2,3,6)$.
\item If $\lambda = 3$, then we have $\nu = 9$ by $3 \lambda \le \nu$ and \eqref{eq:(a-i)sub}.
But then there is no integer $\mu$ satisfying both \eqref{eq:delta(a-i)} and $3 \mu \le 2 \nu - 1$.
\end{itemize}
This completes the proof.
\end{proof}

\subsubsection{\emph{Case (a-ii):} $\wtr (x) \le \wtr (y) < \wtr (z) < \wtr (w)$}

\begin{Lem} \label{lem:nef(a-ii)}
\begin{enumerate}
\item $\Nef (X) = Q + R_z$ and $\nef (X/\mbP^1) = - \lambda - \frac{1}{2} \mu + \nu - 2$.
\item $\delta_X = \lambda + 2 \mu - 2 \nu + 4$.
\end{enumerate}
\end{Lem}

\begin{proof}
Proof can be done in the same as that of Lemma \ref{lem:nef(a-i)}.
Note that, for (ii), we see that $D_z|_X = 2 H_X + \mu F_X$ is base point free and we have $(D_z|_X \cdot \Gamma) = 0$, where $\Gamma = D_x \cdot D_y \cdot X$ is an effective $1$-cycle on $X$.
\end{proof}

\begin{Prop} \label{prop:case(a-ii)}
Suppose that $X/\mbP^1$ does not satisfy the $K^2$-condition.
Then
\[
(\lambda,\mu,\nu) \in \{ (0,1,2), (1,3,5) \}.
\]
\end{Prop}

\begin{proof}
We will classify triplets $(\lambda,\mu,\nu)$ such that $\delta_X > 0$.
In this case, we see $\delta_X \in \mbZ$ by Lemma \ref{lem:nef(a-ii)}.
Hence we assume
\begin{equation} \label{eq:delta(a-ii)}
\lambda + 2 \mu - 2 \nu + 4 \ge 1
\end{equation}
By Lemma \ref{lem:basiccond} and the case assumptions, we have
\begin{equation} \label{eq:basecd(a-ii)}
0  \le 2 \lambda \le \mu -1, \  
3 \mu \le 2 \nu - 1.
\end{equation}
By \eqref{eq:delta(a-ii)} and \eqref{eq:basecd(a-ii)}, we have $\mu \le 3$.
Thus $\mu \in \{1,2,3\}$ since $0 \le \mu - 1$.
\begin{itemize}
\item If $\mu = 1$, then we have $\lambda = 0$ by $0 \le 2 \lambda \le \mu - 1$, and thus we have $\nu = 2$ by $3 \mu \le 2 \nu - 1$ and \eqref{eq:delta(a-ii)}, that is, $(\lambda,\mu,\nu) = (0,1,2)$.
\item If $\mu = 2$, then we have $\lambda = 0$ by $0 \le 2 \lambda \le \mu - 1$.
However there is no integer $\nu$ satisfying both \eqref{eq:delta(a-ii)} and $3 \mu \le 2 \nu - 1$.
\item If $\mu = 3$, then $\mu \ge 5$ by $3 \mu \le 2 \nu - 1$, and we have $\lambda \ge 1$ by $\mu \ge 5$ and  \eqref{eq:delta(a-ii)}.
Thus $\lambda = 1$ by $2 \lambda \le \mu - 1$, and we have $\nu = 5$, that is, $(\lambda,\mu,\nu) = (1,3,5)$.
\end{itemize}
This completes the proof.
\end{proof}

\subsubsection{\emph{Case (b):} $\wtr (w) < \wtr (y)$}

\begin{Lem} \label{lem:restrictb}
Suppose that $\wtr (w) < \wtr (y)$.
Then one of the following holds.
\begin{enumerate}
\item $2 \nu \ge \max \{5 \lambda, 4 \lambda + \mu \}$.
\item $5 \lambda > 2 \nu = 4 \lambda + \mu$.
\item $4 \lambda + \mu > 2 \nu = 5 \lambda$.
\end{enumerate}
\end{Lem}

\begin{proof}
We can write
\[
f = y^5 x a (u,v) + y^4 z b (u,v) + g (u,v,x,y,z,w),
\]
where $a (u,v), b (u,v)$ are of degree $2 \nu - 5 \lambda$ and $2 \nu - 4 \lambda - \mu$, respectively and $g$ is contained in the ideal $(x,z,w)^2$.
Here we understand that $a (u,v) = 0$ if $2 \nu - 5 \lambda < 0$, and $b (u,v) = 0$ if $2 \nu - 4 \lambda - \mu < 0$.
We see that $X$ contains the curve $C_y$, which is the common zero locus of $x, z, w$.

If $a (u,v) = b (u,v) = 0$, then the defining equation of $X$ is $g = 0$ and $g$ has multiplicity $2$ at any point of $C_y \subset X$.
This implies that $X$ is singular along the curve $C_y$.
This is a contradiction and we see that either $2 \nu \ge 5 \lambda$ or $2 \nu \ge 4 \lambda + \mu$.

Suppose that $2 \nu < \max \{5 \lambda,4 \lambda + \mu\}$.
Then either $5 \lambda > 2 \nu \ge 4 \lambda + \mu$ or $4 \lambda + \mu > 2 \nu \ge 5 \lambda$.
If we are in the former case, then $a (u,v) = 0$ and we have $f = y^4 z b + g$. 
If further $\deg b = 2 \nu - 4 \lambda - \mu > 0$, then $X$ has a non-quotient singularity along $C_y \cap (b = 0) \ne \emptyset$.
This is a contradiction and we have $2 \nu = 4 \lambda + \mu$ (and $b$ is a non-zero constant).
If we are in the latter case, a similar argument shows that we have $2 \nu = 5 \lambda$.
This shows that we are in one of the cases (i), (ii) and (iii).
\end{proof}

\begin{Lem} \label{lem:nef(b)}
\begin{enumerate}
\item $\Nef (X) = Q + R_y$ and $\nef (X/\mbP^1) = - \mu + \nu - 2$.
\item $\delta_X = 2 \lambda + \frac{3}{2} \mu - 2 \nu + 4$.
\end{enumerate}
\end{Lem}

\begin{proof}
It is clear that $6 D_y$ is base point free (on $P$) and thus $D_y|_X$ is nef.
It is also clear that the curve $C_y$ is contained in $X$ and $(D_y \cdot C_y) = 0$.
The case assumption $\wtr (w) < \wtr (y)$ implies that the curve $C_y$ is contained in $X$.
Since $(D_y \cdot C_y) = 0$, it follows that $D_y|_X$ is nef and not ample, which proves (i).
This completes the proof.
\end{proof}

\begin{Prop} \label{prop:case(b)}
Suppose that $X/\mbP^1$ does not satisfy the $K^2$-condition.
Then
\[
(\lambda,\mu,\nu) \in \{ (1,-2,1), (2,2,5), (2,3,5), (4,6,10) \}.
\]
\end{Prop}

\begin{proof}
We will classify triplets $(\lambda,\mu,\nu)$ such that $\delta_X > 0$, or equivalently $2 \delta_X \ge 1$.
In the following we assume 
\begin{equation} \label{eq:delta(b)}
2 \delta_X = 4 \lambda + 3 \mu - 4 \nu + 8 \ge 1.
\end{equation}
By Lemma \ref{lem:basiccond} and the case assumptions, we have
\begin{equation} \label{eq:basecd(b)}
0 \le \nu \le 3 \lambda - 1, \ 
3 \mu \le 2 \nu -1.
\end{equation}

Suppose that we are in case (i) of Lemma \ref{lem:restrictb}, that is, the inequalities $2 \nu \ge 5 \lambda$ and $2 \nu \ge 4 \lambda + \mu$ are satisfied.
By \eqref{eq:delta(b)} and $2 \nu \ge 4 \lambda + \mu$, we have $\mu \ge 4 \lambda - 7$.
By \eqref{eq:delta(b)} and $3 \mu \le 2 \nu - 1$, we have $4 \lambda + 5 \ge 3 \mu$.
Combining these inequalities on $\lambda,\mu$, we obtain $\lambda \le 3$.
On the other hand, by $\nu \le 3 \lambda -1$ and $2 \nu \ge 5 \lambda$, we have $\lambda \ge 2$. 
Thus $\lambda \in \{2,3\}$.
\begin{itemize}
\item If $\lambda = 2$, then we have $\nu = 5$ by $2 \nu \ge 5 \lambda$ and $\nu \le 3 \lambda -1$, and we also have $\mu = 2$ by $2 \nu \ge 4 \lambda + \mu$ and \eqref{eq:delta(b)}.
\item If $\lambda = 3$, then we have $\nu = 8$.
However, in this case there is no integer $\mu$ satisfying the inequalities $2 \nu \ge 4 \lambda + \mu$ and \eqref{eq:delta(b)}.
\end{itemize}
Therefore $(\lambda,\mu,\nu) = (2,2,5)$ in this case.

Suppose that we are in case (ii) of Lemma \ref{lem:restrictb}, that is, $5 \lambda > 2 \nu = 4 \lambda + \mu$.
Note that $\mu$ is divisible by $2$.
By $2 \nu = 4 \lambda + \mu$ and \eqref{eq:delta(b)}, we have 
\begin{equation} \label{eq:delta(b)2}
\mu + 7 \ge 4 \lambda.
\end{equation}
By \eqref{eq:delta(b)2} with $5 \lambda > 4 \lambda + \mu$, we have $\mu \le 1$.
On the other hand, by \eqref{eq:delta(b)2} and $2 \mu = 4 \lambda + \mu \ge 0$, we have $\mu \ge -3$.
Since $\mu$ is divisible hy $2$, we have $\mu \in \{-2,0\}$.
\begin{itemize}
\item If $\mu = -2$, then we have $\lambda = 1$ by $4 \lambda + \mu \ge 0$ and \eqref{eq:delta(b)2}, and we have $\nu = 1$.
\item If $\mu = 0$, then we have $\lambda = 1$ by \eqref{eq:delta(b)2} and $5 \lambda > 4 \lambda + \mu$, and we have $\nu = 2$.
However $(\lambda,\mu,\nu) = (1,0,2)$ does not satisfy the inequality $\nu \le 3 \lambda -1$ in \eqref{eq:basecd(b)}.
\end{itemize}
Therefore we have $(\lambda,\mu,\nu) = (1,-2,1)$ in this case.

Finally, suppose that we are in case (iii) of Lemma \ref{lem:restrictb}, that is, $4 \lambda + \mu > 2 \nu = 5 \lambda$.
Note that $\lambda$ is even.
By $2 \nu = 5 \lambda$ and \eqref{eq:delta(b)}, we have
\begin{equation} \label{eq:delta(b)3}
3 \mu + 7 \ge 6 \lambda.
\end{equation}
By $2 \nu = 5 \lambda$ and $\nu \le 3 \lambda -1$, we have $2 \le \lambda$.
By $3 \mu \le 2 \nu - 1 = 5 \lambda -1$ and \eqref{eq:delta(b)3}, we have $\lambda \le 6$.
Since $\lambda$ is even, we have $\lambda \in \{2,4,6\}$.
\begin{itemize}
\item If $\lambda = 2$, then $\nu = 5$ and we have $\mu = 3$ by $3 \mu \le 2 \nu - 1$ and $4 \lambda + \mu > 2 \nu$.
\item If $\lambda = 4$, then $\nu = 10$ and we have $\mu = 6$ by $3 \mu \le 2 \nu - 1$ and \eqref{eq:delta(b)3}.
\item If $\lambda = 6$, then $\nu = 15$ but there is no integer $\mu$ satisfying the inequalities $3 \mu \le 2 \nu - 1$ and \eqref{eq:delta(b)3}.
\end{itemize}
Therefore $(\lambda,\mu,\nu) \in \{ (2,3,5), (4,6,10) \}$ in this case, and the proof is completed.
\end{proof}

\subsection{The classification table and proofs of Theorems \ref{thm:embdP} and \ref{mainthm2}}

We summarize the results of the previous section in Table \ref{table:dP1}.
The computation of $\delta_X$ in each case is done easily by Lemmas \ref{lem:nef(a-i)}, \ref{lem:nef(a-ii)} and \ref{lem:nef(b)}.

\begin{table}[h]
\begin{center}
\caption{$\mathrm{dP}_1$ fibrations not satisfying the $K^2$-condition}
\label{table:dP1}
\begin{tabular}{ccccc}
\hline
No. & $(\lambda,\mu,\nu)$ & $\delta_X$ & Case & $K$-cond. \\
\hline
1 & $(0,-2,0)$ & $1$ & (a-i) & \\
2 & $(0,-1,0)$ & $5/2$ & (a-i) & no \\
3 & $(0,-1,1)$ & $1/2$ & (a-i) & \\
4 & $(0,0,1)$ & $2$ & (a-i) & no  \\
5 & $(1,1,3)$ & $3/2$ & (a-i) & no \\
6 & $(1,2,4)$ & $1$ & (a-i) & \\
7 & $(2,3,6)$ & $1/2$ & (a-i) & \\
8 & $(0,1,2)$ & $2$ & (a-ii) & no \\
9 & $(1,3,5)$ & $1$ & (a-ii) & \\
10 & $(1,-2,1)$ & $1$ & (b) & \\
11 & $(2,2,5)$ & $1$ & (b) &  \\
12 & $(2,3,5)$ & $5/2$ & (b) & no \\
13 & $(4,6,10)$ & $1$ & (b) & 
\end{tabular}
\end{center}
\end{table} 

\begin{Lem} \label{lem:failKcond}
Let $X/\mbP^1$ be a del Pezzo fibration corresponding to a triplet $(\lambda,\mu,\nu)$ in \emph{Table \ref{table:dP1}}.
If $\delta_X > 1$, then $X$ does not satisfy the $K$-condition.
\end{Lem}

\begin{proof}
Suppose that 
\[
(\lambda,\mu,\nu) \in \{ (0,-1,0), (0,0,1), (0,1,2) \}.
\]
By Lemmas \ref{lem:nef(a-i)} and \ref{lem:nef(a-ii)}, we can explicitly compute $\nef (X/\mbP^1)$ and conclude that $\nef (X/\mbP^1) < 0$, which implies that $-K_X$ is ample.
Thus $-K_X \in \Int \Mov (X)$. 

Suppose that $(\lambda,\mu,\nu) \in (1,1,3), (2,3,5)$.
In this case, it is easy to see that $3 D_z$ is a movable divisor on $P$ and we have $\Bs |3 D_z|$ is the closed subset defined by $x = z = 0$.
Since $w^2$ appears in the defining equation of $X$, the codimension of  $X \cap \Bs |3 D_z|$ is $2$ in $X$.
This shows that $D_z|_X$ is movable.
Since $-K_X$ is contained in the interior of the cone generated by $F_X$ and $D_z|_X$, we have $-K_X \in \Int \Mov (X)$.
\end{proof}

\begin{proof}[of \emph{Theorem \ref{thm:embdP}}]
This follows from Propositions \ref{prop:case(a-i)}, \ref{prop:case(a-ii)}, \ref{prop:case(b)} and Lemma \ref{lem:failKcond} (see also Table \ref{table:dP1}).
\end{proof}

\begin{proof}[of \emph{Theorem \ref{mainthm2}}]
The first assertion is Theorem \ref{thm:embdP}.
The rest follows from the first assertion, Theorem \ref{mainthm} and \cite[Theorem 3.3]{Gri01}.
\end{proof}

\begin{Rem} \label{rem:nonsingdP}
A nonsingular del Pezzo fibration of degree $1$ over $\mbP^1$ can be realized as a member of $|6 H + 6 \mu F|$ on a weighted projective space bundle $P (\lambda,2 \mu,3\mu)$ over $\mbP^1$ defined by
\[
\begin{pmatrix}
1 & 1 & 0 & \lambda & 2 \mu & 3 \mu \\
0 & 0 & 1 & 1 & 2 & 3
\end{pmatrix}.
\]
We refer readers to \cite{Gri00} and \cite[Lemma 4.3]{KO} for a classification.
By similar computations given in this subsection, we can conclude that there are only $3$ families of nonsingular del Pezzo fibrations of degree $1$ over $\mbP^1$ such that $(-K_X)^3 + \nef (X/\mbP^1) > 1$ (without assuming $-K_X \notin \Int \Mov (X)$) and they correspond to the pairs 
\[
(\lambda,\mu) = (1,1), (0,1), (2,2).
\]
We have $(-K_X)^3 + \nef (X/\mbP^1) = 3$ (resp.\ $=2$) if $(\lambda,\mu) = (1,1)$ (resp.\ $(\lambda,\mu) = (0,1), (2,2)$).
If $(\lambda,\mu) = (1,1), (0,1)$, then $-K_X \in \Int \Mov (X)$ and $X$ is  birationally non-rigid.
If $(\lambda,\mu) = (2,2)$, then $-K_X \notin \Int \Mov (X)$ and birational rigidity of $X$ is proved in \cite{Gri00}.
\end{Rem}

\section{Comparison of various birational rigidities} \label{sec:variousrig}

We discus subtlety in definitions of birational rigidity of Mori fiber spaces.

\subsection{Several versions of birational rigidity}

We introduce several versions of birational rigidity of Mori fiber spaces, which appear in the literature.

\begin{Def} \label{def:verBR}
We say that a Mori fiber space $X/S$ is {\it birationally rigid} (resp.\ {\it birationally rigid over the base}) if, for any birational map $f \colon X \ratmap X'$ to a Mori fiber space $X'/S'$, there exists a birational automorphism $\alpha \colon X \ratmap X$ (resp.\ birational automorphism $\alpha \colon X \ratmap X$ over the base) such that $f \circ \alpha$ is square.
\end{Def}

We give a formulation of birational rigidity in terms of pliability.
For a Mori fiber space $X/S$, we define
\[
\operatorname{Pli} (X/S) := \{\, Y/T \mid \text{$Y/T$ is a Mori fiber space, $Y$ is birational to $X$}\,\}/\sim_{\operatorname{sq}}
\]
and call it the {\it pliability set} of $X/S$.
Here $\sim_{\operatorname{sq}}$ is the square birational equivalence, that is, $Y_1/T_1 \sim_{\operatorname{sq}} Y_2/T_2$ if and only if there exists a square birational map $Y_1 \ratmap Y_2$.
It is easy to see that $X/S$ is birationally rigid if and only if $\Pli (X/S) = \{[X/S]\}$.

We have the following implications:
\[
\text{BSR} \ \Longrightarrow \ \text{BR over the base} \ \Longrightarrow \text{BR},
\]
where BR and BSR stand for birational rigidity and birational superrigidity, respectively.

As it is explained in Remark \ref{rem:defrig}, for del Pezzo fibrations of degree $1$, BR over the base is equivalent to BSR.
These are no more equivalent for del Pezzo fibrations of degree greater than $1$ as the following example suggests. 

\begin{Example}
Let $X/\mbP^1$ be a nonsingular del Pezzo fibration of degree $2$ and let $C \subset X$ be a section of $X \to \mbP^1$.
Then, by blowing-up $Y \to X$ along $C$, we have a flop $Y \ratmap Y$ and this yields a birational automorphism $\sigma \colon X \ratmap X$ over the base.
Note that the induced birational automorphism between generic fibers of $X \to \mbP^1$ is not biregular, i.e.\ $\sigma$ is not square.
This shows that $X/\mbP^1$ is not birationally superrigid.
However, if in addition $X/\mbP^1$ satisfies the $K^2$-condition, then $X/\mbP^1$ is birationally rigid over the base.
\end{Example}

The most subtle part lies in the comparison of two notions BR and BR over the base. 
In view of the Sarkisov Program (see \cite{Corti95}), an example which separates these notions can occur in the following way.
Suppose that $X/\mbP^1$ is a del Pezzo fibration such that $-K_X$ is nef and big but not ample, and that $X$ admits a flop $\sigma \colon X \ratmap X$.
The flop $\sigma$ is never defined over the base so that $X/\mbP^1$ is not birationally rigid over the base.
If we know that the Sarkisov links from $X/\mbP^1$ other than $\sigma$ are birational automorphisms of $X$, then we can conclude that $X/\mbP^1$ is birationally rigid but not birationally rigid over the base.
Note that this kind of $X/\mbP^1$ does not satisfy the $K$-condition.
In the next subsection, we give a concrete example.

\subsection{Birationally rigid del Pezzo fibrations not satisfying the $K$-condition}

The aim of this subsection is to exhibit an example of a nonsingular del Pezzo fibration $V/\mbP^1$ of degree $1$ such that $V/\mbP^1$ is birationally rigid (in the sense of Definition \ref{def:verBR}) and $V$ fails to satisfy the $K$-condition.

We set $P = P (0,2,3)$, which is a $\mbP (1,1,2,3)$-bundle over $\mbP^1$ defined by
\[
\begin{pmatrix}
1 & 1 & 0 & 0 & 2 & 3 \\
0 & 0 & 1 & 1 & 2 & 3
\end{pmatrix}.
\]
Let $V \in |6 H + 6 F|$ be a member so that the first projection $\pi \colon  V \to \mbP^1$ is a nonsingular del Pezzo fibration of degree $1$ (see Section \ref{sec:toricbundles} for the definitions of $H$ and $F$).
Note that $V/\mbP^1$ is the one corresponding to $(\lambda,\mu) = (0,1)$ in Remark \ref{rem:nonsingdP} and considered in \cite[Proposition 2.12]{Gri00}.

\begin{Prop}[{\cite[Proposition 2.12]{Gri00}}]  \label{prop:birrignoK}
There exists a flop $\tau \colon V \ratmap U$ to a nonsingular del Pezzo fibration $U/\mbP^1$ of degree $1$ in the same family $($i.e.\ $U \in |6 H + 6 F|$$)$.
Moreover, if we are given a birational map $\chi \colon V \ratmap W$ to a Mori fiber space $W/S$, then either $\chi$ or $\chi \circ \tau^{-1} \colon U \ratmap W$ is square.
\end{Prop}

This shows $\operatorname{Pli} (V/\mbP^1) = \{[V/\mbP^1], [U/\mbP^1]\}$, but whether $V/\mbP^1$ is square birational to $U/\mbP^1$ or not is not discussed in \cite{Gri00}.
Note however that this is enough for the purpose of \cite{Gri00}, characterizing birational rigidity over the base in terms of the $K$-condition, because the existence of the flop $\tau$ immediately implies birational non-rigidity of $V/\mbP^1$ over the base. 
In the following, we will construct $V/\mbP^1$ in the family $|6 H + 6 F|$ for which $V \sim_{\operatorname{sq}} U$.

We recall an explicit construction of the flop $\tau \colon V \ratmap U$.
For a sufficiently divisible $k > 0$, the complete linear system $|k (H+F)|$ defines the morphism 
\[
\Phi \colon P \to \mbP := \mbP (1,1,1,1,2,3), \quad
(u\!:\!v ; x\!:\!y\!:\!z\!:\!w) \mapsto (u x\!:\!v x\!:\!u y\!:\!v y\!:\!z\!:\!w).
\]
The image of $\Phi$ is the hypersurface $Q = (s_1 t_2 - s_2 t_1 = 0) \subset \mbP$, where $s_1, s_2, t_1, t_2, z, w$ are the homogeneous coordinates of $\mbP$ of weight $1,1,1,1,2,3$, respectively, and $\Phi \colon P \to Q$ is a birational morphism contracting the surface $S = (x = y = 0) \subset P$ to the curve $\Gamma = (s_1 = s_2 = t_1 = t_2 = 0) \subset Q$.
We define
\[
\begin{split}
\iota_P & \colon P \ratmap P, \quad (u\!:\!v ; x\!:\!y\!:\!z\!:\!w) \mapsto (x\!:\!y ; u\!:\!v\!:\!z\!:\!w), \\
\iota_Q & \colon Q \to Q, \quad (s_1\!:\!s_2\!:\!t_1\!:\!t_2\!:\!z\!:\!w) \mapsto (s_1\!:\!t_1\!:\!s_2\!:\!t_2\!:\!z\!:\!w),
\end{split}
\]
which are birational and biregular involutions, respectively, and they sit in the commutative diagram:
\[
\xymatrix{
P \ar[d]_{\Psi} \ar@{-->}[r]^{\iota_P} & P \ar[d]^{\Psi} \\
Q \ar[r]_{\iota_Q} & Q.}
\]
The map $\iota_P$ is defined outside $S$ and let $U$ be the birational transform of $V$.
We see that $U$ is a member of $|6 H + 6 F|$ and the equation defining $U$ inside $P$ is the one obtained by interchanging $u$ and $x$, and $v$ and $y$ in the equation of $V$.
The image $V' = \Psi (V)$ of $V$ under $\Psi$ is a complete intersection of type $(2, 6)$ in $\mbP$ and $\psi_V = \Psi|_V \colon V \to V'$ is a flopping contraction which contracts the irreducible curve $(x = y = 0) \cap V$ to a point.
We set $U' = \iota_Q (V')$, which is again a complete intersection of type $(2, 6)$ in $\mbP$ and we have $U' = \Psi (U)$.
Again, $\psi_U = \Psi|_U \colon U \to U'$ is a flopping contraction and we have the commutative diagram:
\[
\xymatrix{
V \ar[rd]_{\psi_V} \ar@{-->}[rrr]^{\tau = \iota_P|_U} & & & \ar[ld]^{\psi_U} U \\
& V' \ar[r]^{\cong}_{\iota_Q|_{V'}} & U' &}
\]
This gives the description of the flop $\tau = \iota_P|_V \colon V \ratmap U$ given in Proposition \ref{prop:birrignoK} and we conclude that if $f (u,v,x,y,z,w) = 0$ is the equation for $V$, then $f (x,y,u,v,z,w) = 0$ is the equation for $U$.

We consider special members in $|6 H + 6 F|$ which poses symmetry with respect to the involution $\tau$.
We say that a homogeneous polynomial $f (u,v,x,y,z,w)$ of bi-degree $(6,6)$ is {\it symmetric} if $f (u,v,x,y,z,w) = f (x,y,u,v,z,w)$.
Let $\Lambda \subset H^0 (P, 6 H + 6 F)$ be the vector space consisting of symmetric homogeneous polynomials together with $0$, and let $\mcF$ be the sub linear system of $|6 H + 6 F|$ corresponding to $\Lambda$.
In the following, let $V$ be a general member of $\mcF$ and let $\pi \colon V \to \mbP^1$ be the first projection.

\begin{Lem}
$V$ is nonsingular and $\rho (V) = 2$.
In particular $\pi \colon V \to \mbP^1$ is a nonsingular del Pezzo fibrations of degree $1$.
\end{Lem}

\begin{proof}
We can choose
\[
w^2, \ z^3, \ u^6 x^6, \ v^6 y^6, \ v^6 x^6 + u^6 y^6
\]
as a part of basis of the $\mbC$-vector space $\Lambda$.
Thus the base locus of $\mcF$ is contained in the common zero loci of the above monomials, which is easily seen to be empty.
Thus $\mcF$ is base point free.
The singular locus of $P$ is the set $C_z \cup C_w$, where $C_z = (x = y = w = 0)$ and $C_w = (x = y = z = 0)$ are smooth rational curves on $P$.
By Bertini Theorem, we see that a general member $V \in \mcF$ is nonsingular away from $C_z \cup C_w$.
Since $w^2$ and $z^3$ appear in the defining equation of $V$, we see that $V$ is disjoint from $C_z \cup C_w$.
This shows that $V$ is nonsingular.

Since $\mcF$ is base point free, a general $V \in \mcF$ intersects each torus in $P$ transversally, that is, $V$ is a nondegenerate hypersurface (see \cite[Section 1]{M}).
The morphism defined by $|6 H + 6 F|$ is the birational morphism $\Phi \colon P \to Q$ which contracts the surface $S \subset P$ to the curve $\Gamma \subset Q$.
In particular $\Phi$ does not contract a divisor.
It then follows from \cite[Theorem 3.2]{M} that $\Pic (V)_{\mbC} \cong \Pic (P)_{\mbC}$.
This shows $\rho (V) = \rho (P) = 2$, and thus $V/\mbP^1$ is a del Pezzo fibration of degree $1$.
\end{proof}

\begin{Thm} \label{thm:birrigK}
The del Pezzo fibration $V/\mbP^1$ is birationally rigid and fails to satisfy the $K$-condition.
\end{Thm}

\begin{proof}
The fact that $-K_V \in \Int \Mov (V)$ follows from the existence of the flop $\tau \colon V \ratmap U$. 
Let $\tau \colon V \ratmap U$ be the flop.
Since $V$ is defined by a symmetric homogeneous polynomial, we see that $U = V$.
Thus, Proposition \ref{prop:birrignoK} implies the birational rigidity of $V/\mbP^1$. 
\end{proof}

This tells us that we need to be careful in the definition of birational rigidity when stating Conjecture \ref{conj}.
Subtle behaviors of birational rigidity of Fano varieties are also observed using this kind of symmetries (see \cite{CG} and \cite{O}). 
We give another candidate for separating two notions BR and BR over the base.

\begin{Example}
Let $X$ be a general hypersurface of bi-degree $(2, 6)$ in $\mbP^1 \times \mbP (1,1,2,3)$, which has singular points of type $\frac{1}{2} (1,1,1)$ and $\frac{1}{3} (1,1,2)$.
The first projection $X \to \mbP^1$ is a del Pezzo fibration of degree $1$.
The second projection $X \to \mbP (1,1,2,3)$ is a generically finite morphism of degree $2$ and let $X \xrightarrow{\psi} Z \to \mbP (1,1,2,3)$ be its Stein factorization. 
The morphism $Z \to \mbP (1,1,2,3)$ is a finite morphism of degree $2$ and we have a biregular involution $\iota \colon Z \to Z$ over $\mbP (1,1,2,3)$.
The composite $\psi^{-1} \circ \iota \circ \psi \colon X \ratmap X$ is a flop (see \cite[Section 3.3]{Ot} for a similar construction of $X \ratmap X$).
We see that $(-K_X)^3 + \nef (X/\mbP^1) = 1/3$ is not large and we may expect birational rigidity for $X/\mbP^1$.
However, we cannot apply the arguments in this paper to $X/\mbP^1$ simply because $X$ admits a $\frac{1}{3} (1,1,2)$ point, hence we do not know whether $X/\mbP^1$ is birationally rigid or not.
We leave this as a question: is $X/\mbP^1$ birationally rigid?
\end{Example}


\end{document}